\documentclass[12pt]{article}
\usepackage{color}
\usepackage{url}
\usepackage{hyperref}
\usepackage{amsmath}
\usepackage{amsfonts}
\usepackage{amssymb}
\usepackage{amsthm}
\usepackage{geometry}
\usepackage{enumerate}
\ifdefined\directlua
\usepackage{polyglossia}
\setmainlanguage{english}

\usepackage{fontspec}
\usepackage{microtype}
\usepackage{float}
\restylefloat{table}
\usepackage{pdftexcmds}
\makeatletter
\ifcase\pdf@shellescape
\relax\or
\usepackage{luacode}
\begin{luacode}
  function prtgit()
  local cmd="git show -s --format='%ci %h'"
  print(cmd)
  local r=io.popen(cmd):read("*a")
  if (r) then
  tex.print([[\string\def\string\COMMIT{]]..r..[[}]])
  end
  end
\end{luacode}
\directlua{prtgit()}
\or
\relax\fi
\makeatother
\ifdefined\COMMIT
        \usepackage{background}
        \backgroundsetup{%
         pages=all, placement=bottom,angle=0,scale=1.6,%
         vshift=20pt,hshift=0pt,
         contents={Commit version:\COMMIT}}
\fi
\fi

\newcommand{\FF}{\mathbb F}

\newcommand{\cD}{\mathcal D}

\newcommand{\cL}{\mathcal L}
\newcommand{\cG}{\mathcal G}

\newcommand{\Tr}{\mathrm{Tr}}

\newcommand{\cC}{\mathcal C}

\newcommand{\cH}{\mathcal H}

\newcommand{\Nr}{\mathrm{N}\,}

\newcommand{\PG}{\mathrm{PG}}

\newcommand{\N}{\mathrm{N}}

\newcommand{\fB}{\mathfrak B}

\newcommand{\Hom}{\mathrm{Hom}}

\newcommand{\ind}{\mathrm{ind}\,}

\newtheorem{theorem}{Theorem}[section]
\newtheorem{lemma}[theorem]{Lemma}
\newtheorem{corollary}[theorem]{Corollary}
\newtheorem{proposition}[theorem]{Proposition}

\theoremstyle{definition}
\newtheorem{remark}[theorem]{Remark}

\newtheorem{definition}{Definition}
\title{Identifiers for MRD-codes}
\author{Luca Giuzzi and Ferdinando Zullo\thanks{{The
research  was supported by
Ministry for Education, University and Research of Italy MIUR (Project
PRIN 2012 "Geometrie di Galois e strutture di incidenza") and by the Italian National
Group for Algebraic and Geometric Structures and their Applications (GNSAGA
- INdAM).}}}

\begin{document}
\maketitle
\begin{abstract}
  For any admissible value of the parameters $n$ and $k$ there exist
  $[n,k]$-Maximum Rank Distance $\FF_q$-linear codes. Indeed,
  it can be shown that if field extensions large enough are
  considered, \emph{almost all} rank metric codes are MRD.
  On the other hand, very few families up to equivalence
  of such codes are currently known. In the present paper we
  study some invariants of MRD codes and evaluate their value for
  the known families, providing  a new characterization of
  generalized twisted Gabidulin codes.
\end{abstract}



\section{Introduction}
\label{intro}
Delsarte \cite{Delsarte} introduced  in 1978 rank-distance (RD) codes as
$q$-analogs of the usual linear error correcting codes over finite fields.
%endowed with Hamming distance.
In the same paper, he also showed that the
parameters of these codes must obey a
Singleton-like bound and that for any admissible value of the
length $n$ and the dimension $k$ this bound is sharp.
A rank metric code attaining this bound
is called \emph{maximum rank distance} (MRD).
In 1985  Gabidulin \cite{Gabidulin}  independently rediscovered
Rank-distance codes and also devised an algebraic
decoding algorithm,
in close analogy to what happens for Reed-Solomon codes,
for the family of MRD codes described by Delsarte.

More recently MRD-codes have been intensively investigated both
for their applications to network coding and for their links with
remarkable geometric and algebraic objects such as linear sets and semifields
\cite{BZ,CsMP,CsMPZa,CsMPZ,CsMZ,CsZ,CsZ2,delaCruz,Lu2017,LTZ2,Sheekey,ShVdV}.

It has been shown in \cite{H-TNRR} (see also \cite{BR}) that a generic
rank-distance code, provided that the field involved with the construction
is large enough, is MRD.
The authors of \cite{H-TNRR} make extensive use of algebraic geometry methods
and they are also able to
offer an estimate on the probability that a random
rank metric code is MRD as well as
to show that the probability of obtaining a Gabidulin code in this way
is negligible.

In \cite{BR} Byrne and Ravagnani
obtain an approximation of the fraction
of RM-codes of given length and dimension which are MRD
using a mostly combinatorial
approach. Their paper show also that
some care has to be taken when considering
these density results for codes; indeed, the $\FF_{q^m}$--linear MRD codes
are dense in the family of all $\FF_{q^m}$--linear rank metric codes
$\cC\subseteq\FF_{q^m}^n$ of dimension $k$ and length $n$.
However,
this is not the case for $\FF_q$--linear MRD-codes in the
family of $\FF_q$--linear rank metric codes
$\cC\subseteq\FF_q^{m\times n}$ with $\dim\cC=k$; see \cite{BR}.

In spite of the aforementioned density results, very few families
of MRD codes are currently known up to equivalence; basically,
apart from Gabidulin \cite{Gabidulin} and twisted Gabidulin \cite{Sheekey}
codes, the state of the art is given by the codes
presented in Table~\ref{kMRD} and
their Delsarte duals.

A \emph{distinguisher} for a family of codes $\mathfrak F$ is a polynomial time algorithm which can determine if an arbitrary generator matrix $G$
determines a code belonging to $\mathfrak F$ or not.
Existence of distinguishers is interesting not only as a mean to characterize a code, but also of much importance for applications, since some attacks against McEliece cryptosystems based on it.
The case of (generalized) Gabidulin codes is investigated
in~\cite{H-TM}, whose result we recall in Theorem~\ref{gabidulin-d};
{see also \cite{Neri}, where such codes are characterized in terms of their generator matrices}.

%%%% HERE
% Arguments similar to those of~\cite{H-TNRR} yield that
% the probability of obtaining one of these codes ``at random'' is
% negligible.
The \emph{McEliece cryptosystem} is
a well known and much studied public key cryptosystem
based on error correcting codes.
The basic idea of this encryption scheme is to start with a $t$-error
correcting code endowed with an efficient algorithm for decoding and
hide its generating matrix $G\in\FF_q^{k\times n}$ by
means of a invertible matrix $S$ and a permutation matrix $P$, so that $\hat{G}=SGP$.
Then
 the encryption of a message $m\in\FF_q^k$ is the codeword $c=m\hat{G}+e$
 where $e\in\FF_q^n$ is a noise vector of weight at most $t$.
 In order to discuss the security of this cryptosystem we recall the
 model of \emph{indistinguishability under chosen plaintext attack}
 (IND-CPA). A system is secure under this model if an adversary which does
 not know the key is unable to distinguish between the encodings
 $c_1$ and $c_2$ of any two different messages $m_1$ and $m_2$ she has
 suitably chosen.

 Observe that in the case of McEliece cryptosystem,
 given two distinct messages $m_1,m_2$
 with encodings respectively $c_1$ and $c_2$, the word $c_1-c_2$ has distance
 at most $2t$ from $(m_1-m_2)\hat{G}$. So, when we consider codes
 endowed with Hamming distance and $t$ comparatively ``small'' it is
 easy to see that
 $(m_1-m_2)\hat{G}$ has almost everywhere the
 same components as $c_1-c_2$. This makes IND-CPA easier to thwart.
 Using the rank metric instead of the Hamming metric can improve the
 security. So a potential primary application of MRD codes is for
 McEliece--like cryptosystems.

 Unfortunately,
 even if ``almost all $\FF_{q^m}$--linear codes are MRD'',
 very few of them are known and even less are amenable to efficient decoding.
 The possibility of using Gabidulin codes has been considered in~\cite{H-TM}. The authors in \cite{H-TM} however
 proved that there is a very efficient distinguisher for them; more in detail, it is possible to easily recognize a Gabidulin code from a generic MRD code of the same parameters chosen uniformly at random.
As a consequence, the cryptosystems based on them turn out not to be
semantically secure, as it is possible to distinguish a ciphertext from
a random vector; see also \cite{Overbeck,PRW}.

In the present paper we investigate the existence of algebraic distinguishers
(akin to those of \cite{H-TM}) for the currently known families of $\mathbb{F}_{q^n}$-linear MRD codes and provide some invariants up to equivalence.

Our main results concern the list of dimensions of the intersections
of an $\mathbb{F}_{q^n}$-linear MRD-code with its conjugates and a description of a maximum dimension Gabidulin codes contained in a fixed MRD-code. We shall see, in particular,
that this can be used as to provide distinguishers for the generalized
twisted Gabidulin codes and how it can also be applied to the other $5$
known families (and their duals), see Tables \ref{kMRD} and \ref{DkMRD}.
{We point out that our results answer to the open question \cite[Open Problem II.7.]{B}.}

\subsection{Structure of the paper}
In Section~\ref{prelim} we recall the definitions of rank metric (RM) codes
and their basic properties. We also fix our notation and discuss in
Section~\ref{qpoly} the representation of RM-codes by means of
subspaces of linearized polynomials, the representation
which shall be used in most of
the paper. Section~\ref{vecnum} deals with an alternative convenient
representation of RM-codes.
In Section~\ref{charact} we prove one of our main results, namely
the characterization of generalized twisted Gabidulin codes in terms
of the intersection with their conjugates and the Gabidulin subcode
they contain; see Theorem~\ref{mth1}.
This leads to the introduction in Section~\ref{disting} of
two indexes
\[ h(\cC):=\max\{ \dim(\cC\cap\cC^{[j]})\colon j=1,\ldots,n-1; \gcd(j,n)=1 \}. \]
and
\[ \ind(\cC):=\max\{ \dim \cG : \cG\subseteq \cC \mbox{ is equivalent to a generalized
    Gabidulin code} \} \]
for MRD-codes. These indexes are then evaluated for the known families of
codes. We conclude the paper with some open problems.

\section{Preliminaries}
\label{prelim}
Denote by $\FF_q$ a finite field and let $V_m$ and $V_n$ be
two vector spaces over $\FF_q$ of dimension respectively $m$ and $n$.
The vector space $\Hom_q(V_n,V_m)$ of all $\FF_q$--linear transformations $V_n\to V_m$
is naturally endowed with a \emph{rank distance}
$d_R:\Hom_q(V_n,V_m)\times\Hom_q(V_n,V_m)\to{\mathbb N}$ where $d_R(\varphi,\psi):=\dim \mathrm{Im}(\varphi-\psi)$.
If we fix bases in $V_m$ and $V_n$ we have that
$\Hom_q(V_n,V_m)$ is isometric to the vector space $\FF_q^{m\times n}$ of all
$m\times n$ matrices over $\FF_q$ endowed with the distance
$d(A,B):=rk\,(A-B)$ for all $A,B\in\FF_q^{m\times n}$.

A \emph{rank metric code} or a (also \emph{rank distance code}), in brief RM-code, $\cC$  of parameters
$(m,n,q;d)$ is a subset $\cC$ of $\FF_q^{m\times n}$ with
minimum rank distance $d:=\min_{{A,B \in \cC},\ {A\ne B}} \{ d(A,B) \}$.
A RM-code $\cC$ is $\FF_q$-linear if it is an $\FF_q$-vector subspace of
$\FF_q^{m\times n}$ (or, equivalently, of $\Hom_q(V_m,V_n)$).
When $\cC$ is an $\FF_q$-linear RM-code of dimension $k$
contained in $\FF_q^{n\times n}$, we shall also write,
in brief, that $\cC$ has parameters $[n,k]$.

As mentioned in Section~\ref{intro}, it has been shown in
\cite{Delsarte} that an analogue of the Singleton bound holds for
RM-codes;  namely, if $\cC$ is an $(m,n,q;d)$ RM-code, then
\[ |\cC| \leq q^{\max \{m,n\}(\min \{m,n\}-d+1)}. \]
When this bound is achieved, then $\cC$ is an \emph{MRD-code}.

The \emph{Delsarte dual} code of a linear RM-code $\cC\subseteq  \FF_q^{m \times n}$ is defined as
\[ \cC^\perp=\{ M \in \FF_q^{m \times n} \colon \mathrm{Tr}(MN^t)=0 \hspace{0.1cm}\text{for all}\hspace{0.1cm} N \in \cC\}. \]

\begin{lemma}\label{dualMRD}\cite{Delsarte,Gabidulin}
Let $\cC\subseteq \FF_{q}^{m \times n}$ be an $\FF_q$-linear MRD-code of dimension $k$ with $d>1$. Then the Delsarte dual code $\cC^\perp\subseteq \FF_{q}^{m\times n}$ is an MRD-code of dimension $mn-k$.
\end{lemma}

The weight of a codeword $c\in\cC$ is just the rank of the matrix
corresponding to $c$. The spectrum of weights of a MRD-code is
``complete'' in the following sense (which is a weaker form of \cite[Theorem 5]{Gabidulin}).

\begin{corollary}\cite[Lemma 2.1]{LTZ2}\label{weight}
  Let $\cC$ be an MRD-code in $\FF_q^{m\times n}$ with minimum distance $d$ and
  suppose $m \leq n$. Assume that the null matrix $O$ is in $\cC$.
  Then, for any $0 \leq l \leq m-d$, there
  exists at least
  one matrix $C \in \cC$ such that $\mathrm{rk} (C) = d + l$.
\end{corollary}

Existence of MRD-codes for all possible values $(m,n,q;d)$ of
the parameters has been originally settled in~\cite{Delsarte} where
\emph{Singleton systems} are constructed and,
independently by Gabidulin in~\cite{Gabidulin}; this has also been
generalized in \cite{kshevetskiy_new_2005}.

More recently Sheekey~\cite{Sheekey} discovered a new family of
linear maximum rank metric codes for all possible parameters
which are inequivalent to those above; see also \cite{LTZ}.
Other examples of MRD-codes can be found in \cite{CMP,DS,OOz,OOz2,Sh2018,TZ}.
For some chosen values of parameters there are a few other families of $\mathbb{F}_{q^n}$-linear MRD-codes of $\mathbb{F}_q^{n\times n}$
which are currently known; see~\cite{CsMPZa,CsMPZh,CsMZ}.

The interpretation of linear RM-codes as homomorphisms of vector
spaces prompts the following definition of \emph{equivalence}.
Two RM-codes $\cC$ and $\cC'$ of $\FF_q^{m\times n}$ are \emph{equivalent} if and only if they represent the same homomorphism (up to a change of basis of $V_m$ and $V_n$) in $h\in\Hom_q(V_m,V_n)/\mathrm{Gal}(\FF_{q})$.
This is the same as to say that
there exist two invertible matrices $A\in \FF_q^{m \times m}$, $B\in \FF_q^{n \times n}$ and a field automorphism $\sigma$ such that
$\{A C^\sigma B \colon C\in \cC\}=\cC'$. %In other words, $\cC$ and $\cC'$
%are equivalent if, up to a change of bases they are conjugated as
%sets of homomorphisms.

In general, it is difficult to determine whether two RM-codes are equivalent or not.
The notion of \emph{idealiser} provides an useful criterion.

Let $\cC\subset \FF_q^{m\times n}$ be an RM-code; its left and right idealisers
$L(\cC)$ and $R(\cC)$ are defined as
\[ L(\cC)=\{ Y \in \FF_q^{m \times m} \colon YC\in \cC\hspace{0.1cm} \text{for all}\hspace{0.1cm} C \in \cC\}\]
\[ R(\cC)=\{ Z \in \FF_q^{n \times n} \colon CZ\in \cC\hspace{0.1cm} \text{for all}\hspace{0.1cm} C \in \cC\},\]
see \cite[Definition 3.1]{LN2016}.
These sets appear also in \cite{LTZ2}, where they are
respectively called middle nucleus and right nucleus;
therein the authors prove the following result.
\begin{proposition}\cite[Proposition 4.1]{LTZ2}\label{idealis}
  If $\cC_1$ and $\cC_2$ are equivalent linear RM-codes, then their
  left (resp. right) idealisers are also equivalent.
\end{proposition}
Right idealisers are usually effective as distinguishers for
RM-codes, i.e. non-equivalent RM-codes often have non-isomorphic
idealisers. This is in sharp contrast with the role played by
left idealisers which, for the codes we consider in the present
paper, are always isomorphic to $\FF_{q^n}$.

\subsection{Representation of RM-codes as linearized polynomials}
\label{qpoly}
Any RM-code over $\FF_q$ can
be equivalently defined  either as a subspace of matrices in $\FF_{q}^{m\times n}$
or as a subspace of $\Hom_q(V_n,V_m)$. In the present section we
shall recall a specialized representation in terms of
linearized polynomials which we shall
use in the rest of the paper.

Consider two vector spaces $V_n$ and $V_m$ over $\FF_q$. If $n\geq m$
we can always regard $V_m$ as a subspace of $V_n$ and identify
$\Hom_q(V_n,V_m)$ with the subspace of those $\varphi\in\Hom_q(V_n,V_n)$ such that
$\mathrm{Im}(\varphi)\subseteq V_m$. Also, $V_n\cong\FF_{q^n}$,
when $\FF_{q^n}$ is considered as a $\FF_q$-vector space of dimension $n$.
Let now  $\Hom_q(\FF_{q^n}):=\Hom_q(\FF_{q^n},\FF_{q^n})$ be the set of all
$\FF_q$--linear maps of $\FF_{q^n}$ in itself.
It is well known that each element of $\Hom_q(\FF_{q^n})$ can be
represented in a unique way as a linearized polynomial over $\FF_{q^n}$;
see \cite{lidl_finite_1997}.
In other words, for any $\varphi\in\Hom_q(\FF_{q^n})$ there is
an unique polynomial $f(x)$ of the form
\[ f(x):=\sum_{i=0}^{n-1} a_i x^{q^i}=\sum_{i=0}^{n-1} a_ix^{[i]} \]
with $a_i \in \FF_{q^n}$ and $[i]:=q^i$ such that
\[ \forall x\in\FF_{q^n}:\varphi(x)=f(x). \]
The set $\cL_{n,q}$ of the linearized polynomials over $\FF_{q^n}$ is
a vector space over $\FF_{q^n}$ with respect to the
usual sum and scalar multiplication  of dimension $n$.
When it is regarded as a
vector space over $\FF_q$, its dimension is $n^2$
and it is isomorphic to $\FF_q^{n\times n}$.
We shall use this point of view in the present paper.
Actually, $\cL_{n,q}$ endowed with the product $\circ$ induced by the functional
composition in $\Hom_q(\FF_{q^n})$ is an algebra over $\FF_{q}$.
In particular, given any two
linearized polynomials $f(x)=\sum_{i=0}^{n-1}f_i x^{[i]}$ and
$g(x)=\sum_{j=0}^{n-1}g_j x^{[j]}$, we can write
\[ (f\circ g)(x):=\sum_{i=0}^{n-1}\sum_{j=0}^{n-1} f_ig_j^{[i]} x^{[(i+j)\mathrm{mod}\, n]}. \]

Take now $\varphi\in\Hom_q(\FF_{q^n})$ and let
$f(x)=\sum_{i=0}^{n-1}a_ix^{[i]}\in\cL_{n,q}$
be the associated linearized polynomial.
The \emph{Dickson (circulant) matrix} associated to $f$ is
\[D_f:=
\begin{pmatrix}
a_0 & a_1 & \ldots & a_{n-1} \\
a_{n-1}^{[1]} & a_0^{[1]} & \ldots & a_{n-2}^{[1]} \\
\vdots & \vdots & \vdots & \vdots \\
a_1^{[n-1]} & a_2^{[n-1]} & \ldots & a_0^{[n-1]}
\end{pmatrix}
.\]
It can be seen that the rank of the matrix $D_f$ equals the rank of the $\FF_q$-linear map $\varphi$, see for example \cite{wl} and also \cite{CsMPZ2018,GQ,GS18}.

By the above remarks, it is straightforward to see that any
$\FF_q$-linear RM-code might be regarded as a suitable $\FF_q$-subspace
of $\cL_{n,q}$. This approach shall be extensively used  in the present
paper. In order to fix the notation and ease the reader, we shall
reformulate some of the notions recalled before in terms of linearized polynomials.

A linearized polynomial is called \emph{invertible} if it admits inverse
with respect to $\circ$ or, in other words, if its Dickson matrix
has non-zero determinant. In the remainder of this paper we shall
always silently identify the elements of $\cL_{n,q}$ with the
morphisms of $\Hom_q(\FF_{q^n})$ they represent and, as such,
speak also of \emph{kernel} and \emph{rank} of a polynomial.

Also, two RM-codes $\cC$ and $\cC'$ are equivalent if and only if
there exist two invertible linearized polynomials  $h$ and $g$ and a field automorphism $\sigma$ such that
$\{h \circ f^\sigma \circ g \colon f\in \cC\}=\cC'$.

The notion of Delsarte dual code can be written in terms of
linearized polynomials as follows,
see for example \cite[Section 2]{LTZ}.
Let $b:\cL_{n,q}\times\cL_{n,q}\to\FF_q$ be the bilinear form
given by
\[ b(f,g)=\mathrm{Tr}_{q^n/q}\left( \sum_{i=0}^{n-1} f_ig_i \right) \]
where $\displaystyle f(x)=\sum_{i=0}^{n-1} f_i x^{[i]}$ and $\displaystyle g(x)=\sum_{i=0}^{n-1} g_i x^{[i]} \in \FF_{q^n}[x]$ and
we denote by $\mathrm{Tr}_{q^n/q}$  the trace function $\FF_{q^n}\to\FF_q$
defined as $\mathrm{Tr}_{q^n/q}(x)=x+x^{[1]}+\ldots+x^{[n-1]}$, for $x \in \FF_{q^n}$.
The Delsarte dual code $\cC^\perp$ of a set of linearized polynomials $\cC$ is
\[\cC^\perp = \{f \in \mathcal{L}_{n,q} \colon b(f,g)=0, \hspace{0.1cm}\forall g \in \cC\}. \]

Furthermore, the left and right idealisers of a code
$\cC\subseteq\cL_{n,q}$ can be written as
\[L(\cC)=\{\varphi(x) \in \mathcal{L}_{n,q} \colon \varphi \circ f \in \cC\, \text{for all} \, f \in \cC\};\]
\[R(\cC)=\{\varphi(x) \in \mathcal{L}_{n,q} \colon f \circ \varphi \in \cC\, \text{for all} \, f \in \cC\}.\]

\begin{definition}
  Suppose $\gcd(n,s)=1$ and
  let $\cG_{k,s}:=\langle x^{[0]},x^{[s]},\ldots, x^{[s(k-1)]}\rangle\leq
  \cL_{n,k}$.
  Any code equivalent to $\cG_{k,s}$ is called a
  \emph{generalized Gabidulin code}.
  Any code equivalent to $\cG_{k}:=\cG_{k,1}$ is
  called a \emph{Gabidulin code}.
\end{definition}
\begin{proposition}\cite[Theorem 5]{Sheekey}
  Suppose $\gcd(s,n)=1$ and
  let $\cH_{k,s}(\eta):=\langle x+\eta x^{[sk]}, x^{[s]},\ldots,
  x^{[s(k-1)]}\rangle$.
  If ${\mathrm N}(\eta)={\mathrm N}_{q^n/q}(\eta):=\prod_{i=0}^{n-1}\eta^{[i]}\neq (-1)^{nk}$,
  then $\cH_{k,s}(\eta)$ is a MRD-code with the same parameters as $\cG_{k,s}$.
\end{proposition}
\begin{definition}
Any code equivalent to $\cH_{k,s}(\eta)$ with ${\mathrm N}(\eta)\neq (-1)^{nk}$ and $\eta \neq 0$
is called a \emph{(generalized) twisted Gabidulin code}.
\end{definition}

\begin{remark}\label{k-2}
Clearly, if $1<k<n-1$
\[ \cH_{k,s}(\eta) \cap \cH_{k,s}(\eta)^{[s]}=\langle x^{[2s]},\ldots, x^{[s(k-1)]}\rangle  \]
and so $\dim(\cH_{k,s}(\eta) \cap \cH_{k,s}(\eta)^{[s]})=k-2$ if $\eta \neq 0$.
Indeed, $a_0(x+\eta x^{[sk]})+a_1 x^{[s]}+\ldots+a_{k-1}x^{[s(k-1)]} \in \langle x^{[s]}+\eta^{[s]} x^{[s(k+1)]}, x^{[2s]},\ldots, x^{[sk]}\rangle$ if and only if $a_0=a_1=0$.
\end{remark}

The two families of codes seen above are closed under the Delsarte duality.
\begin{lemma}\cite{Gabidulin,kshevetskiy_new_2005,LTZ,Sheekey}
  The Delsarte dual $\cC^{\perp}$ of an $\FF_{q^n}$-linear MRD-code
  $\cC$ of dimension $k$ is an $\FF_{q^n}$-linear MRD-code of dimension $n-k$.
  Also, $\cG_{k,s}^\perp$ is equivalent to $\cG_{n-k,s}$ and $\cH_{k,s}(\eta)^\perp$ is equivalent to $\cH_{n-k,s}(-\eta^{[n-ks]})$.
\end{lemma}
Apart from the two infinite families of $\FF_{q^n}$-linear MRD-codes $\cG_{k,s}$ and $\cH_{k,s}(\eta)$, there are a few other examples known for $n \in \{6,7,8\}$. Such examples are listed in Table~\ref{kMRD} and
their Delsarte duals in Table~\ref{DkMRD}.

\begin{table}[htp]
\[
  \begin{array}{ |c|c|c|c| }
\hline
\cC & \mbox{parameters} & \mbox{conditions} & \mbox{reference} \\ \hline
\cC_1=\langle x,\delta x^{[1]}+x^{[4]} \rangle_{\FF_{q^6}} & (6,6,q;5) & \begin{array}{cc} q>4 \\ \text{certain choices of} \, \delta \end{array} &  \mbox{\cite[Theorem 7.1]{CsMPZa}} \\ \hline
\cC_2=\langle x, x^{[1]}+x^{[3]}+\delta x^{[5]} \rangle_{\FF_{q^6}} & (6,6,q;5) & \begin{array}{cccc}q \hspace{0.1cm} \text{odd} \\ q \equiv 0,\pm 1 \pmod{5} \\ \delta^2+\delta =1 \\ (\delta \in \FF_q) \end{array} & \mbox{\cite[Theorem 5.1]{CsMPZ}} \\ \hline
\cC_3=\langle x,x^{[s]}, x^{[3s]} \rangle_{\FF_{q^7}} & (7,7,q;5) & \begin{array}{cc} q\,\text{odd}\\ \gcd(s,7)=1 \end{array} &  \mbox{\cite[Theorem 3.3]{CsMPZh}} \\ \hline
\cC_4=\langle x,\delta x^{[1]}+x^{[{5}]} \rangle_{\FF_{q^8}} & (8,8,q;7) & \begin{array}{cc} q\,\text{odd}\\ \delta^2=-1 \end{array} & \mbox{\cite[Theorem 7.2]{CsMPZa}} \\ \hline
\cC_5=\langle x,x^{[s]}, x^{[3s]} \rangle_{\FF_{q^8}} & (8,8,q;6) & \begin{array}{cc} q \equiv 1 \pmod{3} \\ \gcd(s,8)=1 \end{array} & \mbox{\cite[Theorem 3.5]{CsMPZh}} \\ \hline
  \end{array}
\]
\caption{Linear MRD-codes in low dimension}
\label{kMRD}
\end{table}
\begin{table}[htp]
  \[
    \begin{array}{|c|c|c|}
      \hline
\cD_i=\cC_i^{\perp} & \mbox{parameters} & \mbox{conditions}  \\ \hline
    \cD_1=\langle x^{[1]}, x^{[{2}]}, x^{[{4}]},x-\delta^{[{5}]} x^{[{3}]} \rangle_{\FF_{q^6}} & (6,6,q;3) & \begin{array}{cc} q>4 \\ \text{certain choices of} \, \delta \end{array} \\ \hline
\cD_2=\langle x^{[1]},x^{[3]},x-x^{[2]},x^{[4]}-\delta x \rangle_{\FF_{q^6}} & (6,6,q;3) & \begin{array}{cccc}q \hspace{0.1cm} \text{odd} \\ q \equiv 0,\pm 1 \pmod{5} \\ \delta^2+\delta =1\\ (\delta \in \FF_q) \end{array} \\ \hline
\cD_3=\langle x,x^{[{2s}]},x^{[{3s}]},x^{[{4s}]} \rangle_{\FF_{q^7}} & (7,7,q;4) & \begin{array}{cc} q\,\text{odd}\\ \gcd(s,7)=1 \end{array} \\ \hline
\cD_4=\langle x^{[1]},x^{[2]},x^{[3]},x^{[5]},x^{[6]},x-\delta x^{[4]} \rangle_{\FF_{q^8}} & (8,8,q;3) & \begin{array}{cc} q\,\text{odd}\\ \delta^2=-1 \end{array} \\ \hline
\cD_5=\langle x,x^{[{2s}]},x^{[{3s}]},x^{[{4s}]},x^{[{5s}]} \rangle_{\FF_{q^8}} & (8,8,q;4) & \begin{array}{cc} q \equiv 1 \pmod{3} \\ \gcd(s,8)=1 \end{array}
    \\ \hline
  \end{array}
\]
\caption{Delsarte duals of the codes $\cC_i$ for $i=1,\ldots,5$}
\label{DkMRD}
\end{table}

\subsection{Linear RM-codes as subspaces of $\FF_{q^n}^n$}\label{vecnum}

In \cite{Gabidulin}, Gabidulin studied RM-codes as subsets of $\FF_{q^n}^n$. This view is still used in \cite{BR,H-TM,H-TNRR,Neri}.
As noted before, $\mathcal{L}_{n,q}$ equipped with the classical sum and the scalar multiplication by elements in $\FF_{q^n}$ is an $\FF_{q^n}$-vector space.
Let $\mathcal{B}=(g_1,\ldots,g_n)$ an ordered $\FF_q$-basis of $\FF_{q^n}$. The
evaluation mapping
\[ \Phi_{\mathcal{B}}: f(x) \in \mathcal{L}_{n,q} \mapsto (f(g_1),\ldots,f(g_n)) \in \FF_{q^n}^n \]
is an isomorphism between the $\FF_{q^n}$-vector spaces $\mathcal{L}_{n,q}$ and $\FF_{q^n}^n$.
Therefore, if $W$ is an $\FF_{q^n}$-subspace of $\mathcal{L}_{n,q}$, a generator matrix $G$ of $\Phi_{\mathcal{B}}(W)$ can be constructed using the images of a basis of $W$ under the action of $\Phi_{\mathcal{B}}$.
Also, if $G$ is a generator matrix of $\Phi_{\mathcal{B}}(W)$ of maximum rank, then an $\FF_{q^n}$-basis for $W$ can be defined by using the application $\Phi_{\mathcal{B}}^{-1}$ on the rows of $G$.

\section{Characterization of generalized twisted Gabidulin codes}
\label{charact}

A. Horlemann-Trautmann et al. in \cite{H-TM}
proved the following characterization of
generalized Gabidulin codes.
\begin{theorem}[\cite{H-TM}]
  \label{gabidulin-d}
  A MRD-code $\cC$ over $\FF_q$ of length $n$ and dimension $k$ is equivalent
  to a generalized Gabidulin code $\cG_{k,s}$ if and only if there is
  an integer $s<n$ with $\gcd(s,n)=1$ and
  $\dim (\cC\cap\cC^{[s]})=k-1$, where $\cC^{[s]}=\{f(x)^{[s]} \colon f(x)\in \cC\}$.
\end{theorem}
%\begin{corollary}
%  \label{cor57}
%  Let $\cC\subseteq\cL_{n,q}$ be a linear MRD-code of dimension $k<n$.
%  If $\dim (\cC\cap\cC^{[s]})=k-1$ for some $s$ with $\gcd(s,n)=1$, then
%  \[ \cC=\langle x,x^{[s]},\ldots,x^{[s(k-1)]}\rangle\circ p(x) \]
%  for some permutation $q$-polynomial $p(x)$.
%\end{corollary}

If $\cC$ is equivalent to a generalized twisted Gabidulin code $\cH_{k,s}(\eta)$, then $\dim (\cC\cap\cC^{[s]})=k-2$.
This condition, in general, is not
enough to characterize MRD-codes equivalent to $\cH_{k,s}(\eta)$.
The present section is devoted to determine what further
conditions are necessary for a characterization.

Denote by $\tau_\alpha$ the linear application defined by $\tau_\alpha(x)=\alpha x$ and denote by $U_1=\{\Tr\circ\tau_{\alpha} \colon \alpha \in \FF_{q^n}\}=\{\alpha x+ \alpha^{[1]}x^{[1]}+\cdots+\alpha^{[n-1]}x^{[n-1]} \colon \alpha \in \FF_{q^n}\}$. The set $U_1$ is an $\FF_q$-subspace of $\mathcal{L}_{n,q}$ of dimension $n$ whose elements have rank at most one.
It can be proven that the set $\mathcal{U}_1$ of all linearized polynomials with rank at most one is
\[\mathcal{U}_1=\bigcup_{\beta \in \FF_{q^n}^*} \tau_\beta \circ U_1 = \{\tau_\beta \circ \Tr \circ \tau_{\alpha} \colon \alpha,\beta \in \FF_{q^n}\},\]
see e.g. \cite[Proposition 5.1]{LuMaPoTr2014}.

\begin{lemma}
\label{l-basis}
  The space $\cL_{n,q}$ admits a basis of elements contained in
  $U_1$.
\end{lemma}
\begin{proof}
  Let
  $(\alpha_1,\alpha_2,\cdots,\alpha_n)$ be
  a basis of $\FF_{q^n}$ over $\FF_q$.
  Define
  $\overline{\alpha_i}:=\Tr\circ\tau_{\alpha_i}$ for $i=1,\dots,n$
  and $\fB:=\{\overline{\alpha_i}\colon i=1,\dots,n\}$.
  Consider the basis $B_0=(x,x^{[1]},\dots,x^{[n-1]})$ of $\cL_{n,q}$;
  the components of the vectors of $\fB$ with respect to this
  basis are the rows of the following matrix
  \[ M:=\begin{pmatrix}
      \alpha_1 & \alpha_1^{[1]} & \dots & \alpha_1^{[n-1]} \\
      \alpha_2 & \alpha_2^{[1]} & \dots & \alpha_2^{[n-1]} \\
      \vdots    &                &       & \vdots \\
      \alpha_n & \alpha_n^{[1]} & \dots & \alpha_n^{[n-1]} \\
    \end{pmatrix} \]
  By ~\cite[Corollary 2.38]{lidl_finite_1997}, $\det(M)\neq 0$;
  in particular the vectors of $\fB$ are linearly independent
  in $\cL_{n,q}$ and so $\fB$ is a basis for $\cL_{n,q}$.
\end{proof}

\begin{lemma}\label{trace}
Let $n$ and $s$ be two integers such that $\gcd(s,n)=1$, if $p(x) \in \cL_{n,q}$ and $p(x)=\lambda p(x)^{[s]}$ for some $\lambda \in \FF_{q^n}^*$, then $p(x)$ is in $\mathcal{U}_1$.
\end{lemma}
\begin{proof}
  Under the assumptions, the map $x\to x^{[s]}$ is a generator of
  the Galois group of $\FF_{q^n}:\FF_q$. In particular, for all $0\leq i\leq n-1$
  there are $\lambda_i$ such that
  $p(x)=\lambda_i p(x)^{[i]}$. It follows that the Dickson matrix of
  $p(x)$ has rank at most $1$ and this proves the thesis.
\end{proof}

For the sake of completeness we prove the following lemma; see also~\cite[Lemma 3]{Lun99}.

\begin{lemma}\label{fixedspace}
  Let $n$ and $s$ be two integers such that $\gcd(s,n)=1$, if $W \neq \{0\}$ is an $\FF_{q^n}$-subspace of $\mathcal{L}_{n,q}$ such that $W= W^{[s]}$, then $W$ admits a basis of vectors in ${U}_1$.
  \end{lemma}
  \begin{proof}
    By Lemma~\ref{l-basis}, there exists a
    basis $\fB$ of $\cL_{n,q}$ consisting of vectors of ${U}_1$.
    In particular, for any $b\in\fB$ and $i=0,\dots,n-1$ we have $b^{[i]}=b$.
    There is a unique matrix $G$ in row reduced echelon form whose
    rows contain the components of a basis of $W$ with respect to the basis
    $\fB$.
    Since $W=W^{[s]}$ and $\fB^{[s]}=\fB$ we have that
    the rows of $G^{[s]}$ contain the components of a basis of
    $W^{[s]}$ with respect to $\fB$; however $G^{[s]}$ represents
    also the vectors of a basis of $W$ and it is in
    row reduced echelon form; so $G^{[s]}=G$.
    Since $\gcd(s,n)=1$
    this yields that all entries of $G^{[s]}$ are defined over
    $\FF_q$. In particular, each vector of this basis of $G$ is
    in the vector space $U_1$ over $\FF_q$, that is it has
    rank $1$.
  \end{proof}

The following Lemma rephrases the requirements of
Theorem~\ref{gabidulin-d} in a more suitable way for the arguments
to follow.
\begin{lemma}\label{Gablemma}
Let $n$ and $s$ be two integers such that $\gcd(s,n)=1$ and let $\cC$ be an $\FF_{q^n}$-subspace of dimension $k>1$ of $\cL_{n,q}$.
If $\dim(\cC \cap \cC^{[s]})=k-1$ and $\cC \cap U_1 = \{0\}$, then there exists $p(x)$ such that
\[ \cC=\langle p(x),p(x)^{[s]},\ldots,p(x)^{[{s(k-1)}]} \rangle_{\FF_{q^n}}. \]
If $\cC$ contains at least one
invertible linearized polynomial, then $p(x)$ is invertible and $\cC\cong\cG_{k,s}$.
\end{lemma}
\begin{proof}
  Note that, since $\cC$ is an $\FF_{q^n}$-subspace and $\cC \cap U_1=\{0\}$, then $\cC\cap\mathcal{U}_1=\{0\}$. We argue by induction.
  We first prove the case $k=2$. By hypothesis, $\cC\cap \cC^{[s]}=\langle h(x) \rangle$ and so $h(x)^{[s]} \in \cC^{[s]}$.
  Since $\cC \cap\mathcal{U}_1=\{0\}$, by Lemma \ref{trace} the polynomials $h(x)$ and $h(x)^{[s]}$ are linearly independent over $\FF_{q^n}$ and $\cC=\langle h(x)^{[s(n-1)]}, h(x) \rangle_{\FF_{q^n}}=\langle p(x), p(x)^{[s]} \rangle_{\FF_{q^n}}$, with $p(x)=h(x)^{[{s(n-1)}]}$.
%For the second part, if $\cC$ contains at least one invertible linearized polynomial, then $p(x)$ has to be invertible. Indeed, if there exists $x_0 \in \FF_{q^n}^*$ such that $p(x_0)=0$, then $\alpha p(x_0)+\beta p(x_0)^{q^s}=0$ for each $\alpha, \beta \in \FF_{q^n}$, which is not possible.

Suppose now that the assert holds true for $k-1$ and take $k>2$.
Let $V:=\cC\cap\cC^{[s]}$, $V$ is an $\FF_{q^n}$-subspace of $\cC$ of dimension $k-1$ such that $V \cap \mathcal{U}_1=\{0\}$, hence by Lemma \ref{fixedspace} $V\neq V^{[s]}$.
Then, since $V$ and  $V^{[s]}$ are both contained in $\cC^{[s]}$,
by Grassmann's formula
\[ \dim (V\cap V^{[s]})=k-2. \]
So, $\dim V=k-1$, $V\cap \mathcal{U}_1 =\{0\}$ and $\dim (V \cap V^{[s]})=k-2$. By induction, there is $h(x) \in V$ such that
\[V=\langle h(x), h(x)^{[s]},\ldots, h(x)^{[{s(k-2)}]} \rangle_{\FF_{q^n}}.\]
Also,
\[ h(x)^{[{s(n-1)}]} \in V^{[{s(n-1)}]}=\cC^{[{s(n-1)}]}\cap\cC \subset \cC.\]
If it were $h(x)^{[{s(n-1)}]} \in V$, then $V=V^{[s]}$,
which has already been excluded.
%which is ruled out by Lemma \ref{fixedspace}.
So,
\[ \cC=\langle p(x), p(x)^{[s]}, \ldots, p(x)^{[{s(k-1)}]} \rangle_{\FF_{q^n}}, \]
where $p(x)=h(x)^{[{s(n-1)}]}$.

Suppose now there is $x_0\in\FF_{q^n}^*$ such that $p(x_0)=0$. Then,
$\alpha_1p(x_0)+\cdots+\alpha_kp(x_0)^{[s(k-1)]}=0$ for any choice
of $\alpha_i\in\FF_{q^n}$, $i=1,\ldots,k$.
In particular, if $\cC$ contains at least
one invertible linearized polynomial, then $p(x)$ must also be
invertible.
In such a case
\[ \cC \circ p^{-1}(x)=\langle x, x^{[s]}, \ldots, x^{[{s(k-1)}]} \rangle_{\FF_{q^n}};\]
so $\cC$ is equivalent to $\mathcal{G}_{k,s}$.
\end{proof}

We now focus on the case $\dim (\cC \cap \cC^{[s]})=k-2$.
\begin{itemize}
  \item
    If $\dim \cC=2$ we just have
$\cC=\langle p(x), q(x)\rangle_{\FF_{q^n}}$ with $q(x) \notin \langle p(x)^{[s]}\rangle_{\FF_{q^n}}$ and $p(x) \notin \langle q(x)^{[s]}\rangle_{\FF_{q^n}}$.
\item
Suppose $\dim \cC=3$, $\dim (\cC \cap \cC^{[s]})=1$ and $\cC \cap \mathcal{U}_1=\{0\}$.
As before, write $V:=\cC \cap \cC^{[s]}$.
Since $V$ and $V^{[s]}$ are contained in $\cC^{[s]}$, by Grassmann's formula,
\[ 0\leq \dim(V \cap V^{[s]}) \leq 1. \]
So, either $V=V^{[s]}$ or $\dim (V \cap V^{[s]})=0$.
The former case is ruled out by Lemma~\ref{fixedspace}.
So $\dim (V \cap V^{[s]})=0$ and $V=\langle h(x) \rangle_{\FF_{q^n}}$.
It follows that
\[\cC= \langle p(x), p(x)^{[s]} \rangle_{\FF_{q^n}} \oplus \langle q(x) \rangle_{\FF_{q^n}},\]
with $p(x)=h(x)^{[{s(n-1)}]}$.
\item
Suppose that $\dim \cC=4$, $\dim (\cC \cap \cC^{[s]})=2$ and $\cC \cap \mathcal{U}_1=\{0\}$. Write $V:=\cC \cap \cC^{[s]}$.
Clearly, since $V\neq V^{[s]}$,
\[ 0 \leq \dim(V \cap V^{[s]}) \leq 1. \]

Suppose $\dim (V \cap V^{[s]})=1$.
Then, the subspace $V$ fulfills all of the assumptions of Lemma~\ref{Gablemma},
so there is
$h(x) \in V$ such that
\[ V=\langle h(x), h(x)^{[s]} \rangle_{\FF_{q^n}} \]
and $h(x)^{[{s(n-1)}]}\in \cC \setminus V$, since, otherwise, $V=V^{[s]}$.
So,
\[ \cC = \langle p(x), p(x)^{[s]}, p(x)^{[{2s}]} \rangle_{\FF_{q^n}} \oplus \langle q(x) \rangle_{\FF_{q^n}}, \]
with $p(x)=h(x)^{[{s(n-1)}]}$.

Suppose now that $\dim(V \cap V^{[s]})=0$; then
$\cC= V \oplus V^{[{s(n-1)}]}$.
If $V=\langle h(x), g(x) \rangle_{\FF_{q^n}}$ then $V^{[{s(n-1)}]}=\langle h(x)^{[{s(n-1)}]}, g(x)^{[{s(n-1)}]} \rangle_{\FF_{q^n}}$,
and so
\[ \cC=\langle p(x), p(x)^{[s]}\rangle_{\FF_{q^n}} \oplus \langle q(x), q(x)^{[s]}\rangle_{\FF_{q^n}}, \]
with $p(x)=h(x)^{[{s(n-1)}]}$ and $q(x)=g(x)^{[{s(n-1)}]}$.
\end{itemize}

More generally, we can prove the following result.

\begin{theorem}
\label{mth1}
  Let $n$ and $s$ be two integers such that $\gcd(s,n)=1$ and let $\cC$ be an $\FF_{q^n}$-subspace of dimension $k>2$ of $\cL_{n,q}$.
  Let $V:=\cC \cap \cC^{[s]}$.
  Suppose $\dim V=k-2$ and $\cC \cap \mathcal{U}_1 =\{0\}$, then $\cC$ has one of the following forms
\begin{enumerate}
  \item if $\dim (V \cap V^{[s]})=k-3$, then there exist $p(x)$ and $q(x)$ in $\cC$ such that
        \[ \cC = \langle p(x), p(x)^{[s]}, \ldots, p(x)^{[{s(k-2)}]} \rangle_{\FF_{q^n}} \oplus \langle q(x) \rangle_{\FF_{q^n}}; \]
  \item if $\dim (V \cap V^{[s]})=k-4$, then there exist $p(x)$ and $q(x)$ in $\cC$ such that
        \[ \cC = \langle p(x), p(x)^{[s]}, \ldots, p(x)^{[{s(i-1)}]} \rangle_{\FF_{q^n}} \oplus \langle q(x), q(x)^{[s]}, \ldots, q(x)^{[{s(j-1)}]} \rangle_{\FF_{q^n}}, \]
        where $i+j=k$ and $i,j \geq 2$.
\end{enumerate}
\end{theorem}
\begin{proof}
  We have already proved the assert for $k \leq 4$.
  Assume by induction that the assert holds for each $t <k$ with $k\geq 4$.
Since $V$ and $V^{[s]}$ are contained in $\cC^{[s]}$, it follows that
\[ \dim (V\cap V^{[s]})\geq k-4, \]
that is  $\dim (V\cap V^{[s]}) \in \{k-4,k-3\}$, since $V \neq V^{[s]}$.
If $\dim (V\cap V^{[s]})=k-3$, then, by Lemma \ref{Gablemma}, there exists $h(x) \in V$ such that
\[ V=\langle h(x), h(x)^{[s]}, \ldots, h(x)^{[{s(k-3)}]} \rangle_{\FF_{q^n}}.\]
Since $h(x)^{[{s(n-1)}]}\in \cC \setminus V$ (otherwise $V=V^{[s]}$), we get
\[ \cC=\langle p(x), p(x)^{[s]},\ldots, p(x)^{[{s(k-2)}]} \rangle_{\FF_{q^n}}\oplus \langle q(x) \rangle_{\FF_{q^n}}, \]
where $p(x)=h(x)^{[{s(n-1)}]}$.
If $\dim (V \cap V^{[s]})=k-4$, since $V$ has dimension $k-2$ and $V\cap \mathcal{U}_1=\{0\}$, by induction there exist $h(x)$ and $g(x)$ such that
either
\[ V=\langle h(x),\ldots,h(x)^{[s[k-4]]}\rangle_{\FF_{q^n}} \oplus \langle g(x) \rangle_{\FF_{q^n}} \]
or
\[ V=\langle h(x),\ldots, h(x)^{[{s(l-1)}]} \rangle_{\FF_{q^n}} \oplus \langle g(x),\ldots, g(x)^{[{s(m-1)}]} \rangle_{\FF_{q^n}}, \]
with $l+m=k-2$.
Since $V, V^{[{s(n-1)}]} \subset \cC$ and $\dim V\cap V^{[s]}=k-4$
we get $\cC=V + V^{[{s(n-1)}]}$.
So,
either
\[ \cC=\langle h(x)^{[{s(n-1)}]}, h(x), \ldots, h(x)^{[{s(k-4)}]} \rangle_{\FF_{q^n}} \oplus \langle g(x)^{[{s(n-1)}]}, g(x) \rangle_{\FF_{q^n}} \]
or
\[ \cC=\langle h(x)^{[{s(n-1)}]}, h(x), \ldots, h(x)^{[{s(l-1)}]} \rangle_{\FF_{q^n}} \oplus \langle g(x)^{[{s(n-1)}]}, g(x), \ldots, g(x)^{[{s(m-1)}]} \rangle_{\FF_{q^n}}. \]
If we now put $p(x)=h(x)^{[{s(n-1)}]}$ and $q(x)=g(x)^{[{s(n-1)}]}$,
then we get the assert.
\end{proof}

Examples of $k$-dimensional MRD-codes $\cC$ with $\dim(\cC\cap\cC^{[s]})=k-2$ and $\dim (V\cap V^{[s]})=k-3$, where $V=\cC\cap\cC^{[s]}$,
are the generalized twisted Gabidulin codes; see Remark \ref{k-2}.
An example where $\dim V\cap V^{[s]}=k-4$
is given by the code $\mathcal{D}_2$ (see Table \ref{DkMRD}), which can be written as
\[ \mathcal{D}_2=\langle -x+x^{[2]} , -x^{[1]}+x^{[3]} \rangle_{\FF_{q^6}} \oplus \langle -\delta x^{[1]} +x^{[3]}, -\delta x^{[2]}+x^{[4]} \rangle_{\FF_{q^6}}. \]

%%% EXTRA TYPE

\begin{lemma}
  \label{ll0}
Let $\cC\subseteq \cL_{n,q}$ be an $\FF_{q^n}$-linear RM-code  with dimension $k$ containing a MRD-code $\cG$ equivalent to a generalized Gabidulin code $\cG_{l,s}$ of dimension $l \leq k$, then there exists a permutation linearized polynomial $p(x)$ and $(k-l)$ linearized polynomials $q_1(x),\ldots,q_{k-l}(x)$ such that
\begin{equation}\label{form}
  \cC=\langle q_1(x),\ldots,q_{k-l}(x), p(x), p(x)^{[s]},\ldots p(x)^{[s(l-1)]}
  \rangle.
\end{equation}
\end{lemma}
We call the polynomials $q_i(x)$ of Lemma~\ref{ll0}
\emph{polynomials of extra type}.
\begin{proof}
By Lemma \ref{Gablemma}, there exists a permutation linearized polynomial $p(x)$ such that
\[\mathcal{G}= \langle p(x), p(x)^{[s]},\ldots,p(x)^{[s(l-1)]} \rangle_{\FF_{q^n}} \]
and $p(x)$,\dots,$p(x)^{[s(l-1)]}$ are linearly independent.
Now, we can extend the list of polynomials
$\{p(x), p(x)^{[s]},\ldots,p(x)^{[s(l-1)]}\}$ to a basis of $\cC$ with
suitable polynomials $q_i$ as to get the form~\eqref{form}.
\end{proof}

\begin{lemma}\label{charH}
If $\cC \subseteq \cL_{n,q}$ is an $\FF_{q^n}$-linear MRD-code of dimension $k$ containing a code $\cG$ equivalent to $\cG_{k-1,s}$, i.e. $\cG=\langle p(x), p(x)^{[s]},\ldots, p(x)^{[s(k-2)]} \rangle_{\FF_{q^n}}$ with $p(x)$ an invertible linearized polynomial, and for which there exists an extra polynomial $g(x)$ in $\langle p(x)^{[-s]},p(x)^{[s(k-1)]}\rangle_{\FF_{q^n}}$ with $g(x)=p(x)^{[-s]}+\eta p(x)^{[s(k-1)]}$ and $\N_{q^n/q}(\eta)\neq (-1)^{kn}$, then $\cC$ is equivalent to $\cH_{k,s}(\eta^{[s]})$.
\end{lemma}
\begin{proof}
By the previous lemma and by hypothesis,
\[ \cC=\langle g(x), p(x), p(x)^{[s]},\ldots, p(x)^{[s(k-2)]} \rangle_{\FF_{q^n}}, \]
with $p(x)$ permutation linearized polynomial and $g(x) = p(x)^{[-s]}+\eta p(x)^{[s(k-1)]}$.
Since $\cC$ and $\cC^{[s]}$ are equivalent, we can suppose that
\[ \cC=\langle q(x), p(x)^{[s]},\ldots, p(x)^{[s(k-1)]} \rangle_{\FF_{q^n}}, \]
with $q(x) = p(x)+\eta^{[s]} p(x)^{[sk]}$.
So,
\[ \cC=\langle x^{[s]},\ldots,x^{[{s(k-1)}]}, x+\eta^{[s]} x^{[sk]} \rangle_{\FF_{q^n}} \circ p(x). \]
Since $\cC$ is a MRD-code, then $\cC\circ p^{-1}(x)=\langle x^{[s]},\ldots,x^{[{s(k-1)}]}, x+\eta^{[s]} x^{[sk]} \rangle_{\FF_{q^n}}$ is also a MRD-code equivalent to $\cH_{k,s}(\eta^{[s]})$.
\end{proof}

Theorem~\ref{mth1} prompts the following characterization of generalized twisted Gabidulin codes.

\begin{theorem}
  Let $\cC$ be an $\FF_{q^n}$-linear MRD-code of dimension $k>2$ contained
  in $\cL_{n,q}$.
  Then, the code $\cC$ is equivalent to a generalized twisted Gabidulin code if and only if there exists an integer $s$ such that $\gcd(s,n)=1$ and the following two conditions hold
\begin{enumerate}
  \item $\dim (\cC \cap \cC^{[s]})=k-2$ and $\dim(\cC\cap \cC^{[s]} \cap \cC^{[{2s}]})=k-3$, i.e.
        there exist $p(x),q(x) \in \cC$ such that
        \[ \cC= \langle p(x)^{[s]}, p(x)^{[{2s}]}, \ldots, p(x)^{[{s(k-1)}]} \rangle_{\FF_{q^n}} \oplus \langle q(x) \rangle_{\FF_{q^n}}; \]
  \item $p(x)$ is invertible and there exists $\eta \in \FF_{q^n}^*$ with $\Nr_{q^n/q}(\eta)\neq (-1)^{kn}$ such that $p(x)+\eta p(x)^{[{sk}]} \in \cC$.
\end{enumerate}
\end{theorem}
\begin{proof}
  The proof follows directly
  from Theorem~\ref{mth1}
  and Lemma~\ref{charH}.
\end{proof}

As a consequence we get the following.

\begin{theorem}
Let $\cC$ be an $\FF_{q^n}$-linear RM-code of dimension $k>2$ of $\cL_{n,q}$, with $\cC \cap \mathcal{U}_1=\{0\}$.
If there exists an integer $s$ such that $\gcd(s,n)=1$ and
\begin{enumerate}
  \item $\dim (\cC \cap \cC^{[s]})=k-2$ and $\dim(\cC\cap \cC^{[s]} \cap \cC^{[{2s}]})=k-3$, i.e.
    there exist $p(x),q(x) \in \cC$ such that
    \[ \cC= \langle p(x)^{[s]}, p(x)^{[{2s}]}, \ldots, p(x)^{[{s(k-1)}]} \rangle_{\FF_{q^n}} \oplus \langle q(x) \rangle_{\FF_{q^n}}; \]
  \item $p(x)$ is invertible and there exists $\eta \in \FF_{q^n}^*$ such that $p(x)+\eta p(x)^{[{sk}]} \in \cC$ and $\Nr_{q^n/q}(\eta)\neq (-1)^{kn}$,
\end{enumerate}
then $\cC$ is a MRD-code equivalent to $\cH_{k,s}(\eta)$.
\end{theorem}

Note that if such invertible linearized polynomial $p(x)$ exists, then $\cC\cap\cC^{[s]}\cap\cdots\cap\cC^{[{s(k-2)}]}=\langle p(x)^{[{s(k-2)}]} \rangle_{\FF_{q^n}}$.

\section{Distinguishers for RM-codes}
\label{disting}
A \emph{distinguisher} is an easy to compute function which allows to
identify an object in a family of (apparently) similar ones.
Existence of distinguishers is of particular interest for cryptographic
applications, as it makes possible to identify a candidate encryption
from a random text.

As seen in the previous section,
it has been shown in~\cite{H-TM} that an MRD-code $\cC$ of parameters
$[n,k]$ is equivalent
to a generalized Gabidulin code if, and only if, there exists a positive integer $s$ such that $\gcd(s,n)=1$ and $\dim(\cC\cap\cC^{[s]})=k-1$.
Following the approach of \cite{H-TM}, we define for any RM-code
$\cC$ the number
\[ h(\cC):=\max\{ \dim(\cC\cap\cC^{[j]})\colon j=1,\ldots,n-1; \gcd(j,n)=1 \}. \]
Theorem~\ref{gabidulin-d} states that
an MRD-code $\cC$ is equivalent to a generalized Gabidulin code
if and only if $h(\cC)=k-1$.

Also, for any given $\FF_{q^n}$-linear code $\cC$, the following
proposition is immediate.
\begin{proposition}
  For any  $k$-dimensional $\FF_{q^n}$-linear code $\cC$,
  %of $\cL_{n,q}$, then
  \[\cC^{[i]\perp}=\cC^{\perp[i]},\]
  for each $i\in \{0,\ldots,n-1\}$.
  So, we have
\[ h(\cC^\perp)=n-2k+h(\cC). \]
\end{proposition}

We now define also the \emph{Gabidulin index}, $\ind(\cC)$ of a $[n,k]$ RM-code
as the maximum dimension of a subcode $\cG\leq\cC$ contained in $\cC$ with
$\cG$ equivalent to a generalized Gabidulin code.

Clearly, $1\leq\ind(\cC)\leq k$ and $\ind(\cC)=k$ if and only if $\cC$
is a Gabidulin code.
It can be readily seen that if $\cC$ and $\cC'$ are two equivalent codes,
then they have the same indexes $\ind(\cC)=\ind(\cC')$ and $h(\cC)=h(\cC')$.
Also, $h(\cC)\geq\ind(\cC)-1$ for RM-codes.

We shall now prove that for the known codes the Gabidulin index can be
effectively computed.
More in detail, in the next theorem we
determine these indexes for each known $\FF_{q^n}$-linear MRD-code.
Our result is contained in Table~\ref{tab1}. Also in the table
we recall the right idealisers (up to equivalence) for these codes (see also \cite{PhDthesis}).

\begin{theorem}
  The Gabidulin indexes $\ind(\cC)$ and the values of $h(\cC)$ for the known MRD-codes $\cC$ of parameters $[n,k]$ are
  as given in Table~\ref{tab1}.
  \begin{table}[htp]
    \[
      \begin{array}{ |c|c|c|c|c| }
        \hline
\mbox{Code } & \ind & h & R & [n,k] \\ \hline
\cG_{k,s} & k & k-1 & \FF_{q^n} & [n,k] \\ \hline
\cH_{k,s}(\eta) & k-1 & k-2 & \FF_{q^{\gcd(n,k)}} & [n,k] \\ \hline
\cC_1 & 1 & 0 & \FF_{q^3} & [6,2] \\ \hline
\cC_2 & 1 & 0 & \FF_{q^2} & [6,2] \\ \hline
\cC_3 & 2 & 1 & \FF_{q^n} & [7,3] \\ \hline
\cC_4 & 1 & 0 & \FF_{q^4} & [8,2] \\ \hline
\cC_5 & 2 & 1 & \FF_{q^n} & [8,3] \\ \hline
      \end{array}\qquad\qquad
      \begin{array}{ |c|c|c|c|c| }
        \hline
\mbox{Code } & \ind & h & R & [n,k] \\ \hline
 & & & & \\ \hline
 & & & & \\ \hline
\cD_1 & 2 & 2 & \FF_{q^3} & [6,4] \\ \hline
\cD_2 & 2 & 2 & \FF_{q^2} & [6,4] \\ \hline
\cD_3 & 3 & 2 & \FF_{q^n} & [7,4] \\ \hline
\cD_4 & 3 & 4 & \FF_{q^4} & [8,6] \\ \hline
\cD_5 & 4 & 3 & \FF_{q^n} & [8,5] \\ \hline
\end{array}
\]
\caption{Known linear MRD-codes and their Gabidulin index}
\label{tab1}
\end{table}
  \end{theorem}
  \begin{proof}
    Clearly, the Gabidulin index of a generalized Gabidulin code is $k$;
    any twisted generalized Gabidulin code of dimension $k$ contains
    a generalized Gabidulin code of dimension $k-1$;
    so, its index is $k-1$.

    We now consider the case of the codes
    $\cC_1, \cC_2, \cC_3, \cC_4, \cC_5, \cD_3$ and $\cD_5$.
    By construction, it is immediate to see that they all contain a
    generalized Gabidulin code of codimension $1$; so, they also
    have Gabidulin index $k-1$, where $k$ is the dimension of the code.
    Also for all of them $k-2\leq h(\cC)<k-1$, so $h(\cC)=k-2$.

    The cases of the dual codes $\cD_{i}$ with $i=1,2,4$ must
    be studied in more detail.
    First we prove that
    the codes $\cD_1, \cD_2$ and $\cD_4$ do not contain any code equivalent to $\cG_{k-1,s}$, for any $s$, i.e. that their Gabidulin index is less than $k-1$
    and then determine the exact value.

\begin{flushleft}
  \textbf{The code $\cD_1$}
\end{flushleft}

By Table~\ref{DkMRD}, we have that
\[ \cD_1=\langle x^{[1]}, x^{[{2}]}, x^{[{4}]},x-\delta^{[{5}]} x^{[{3}]} \rangle_{\FF_{q^6}}. \]
Suppose that there is a code $\overline{\cD}$ contained in $\cD_1$
equivalent to a generalized Gabidulin code of dimension $3$, i.e. either $\overline{\cD}\simeq \cG_{3,1}$ or $\overline{\cD}\simeq \cG_{3,5}$.
Since $\cG_{3,1}$ and $\cG_{3,5}$ are equivalent, then $\overline{\cD}$
is equivalent to $\cG_{3,1}$.
By Theorem~\ref{gabidulin-d},
$h({\overline{\cD}})=2$; on the other hand, since $\cD_1$ is
not equivalent to a Gabidulin code it must be $h(\cD)< 3$.
So, $\cD_1\cap \cD_1^{[1]} = \overline{\cD}\cap \overline{\cD}^{[1]}$ and hence $\cD_1\cap\cD_1^{[5]}=\overline{\cD}\cap\overline{\cD}^{[5]}$. From these equalities we get
\[ \overline{\cD}\cap \overline{\cD}^{[1]}= \langle x^{[{2}]}, x^{[1]}-\delta x^{[4]}\rangle_{\FF_{q^6}} \]
and
\[ \overline{\cD}^{[5]}\cap \overline{\cD}=\langle x^{[1]}, x-\delta^{[5]} x^{[3]} \rangle_{\FF_{q^6}}. \]
Since $\dim \overline{\cD}=3$ we obtain
\[ \overline{\cD}= \langle x^{[1]}, x^{[{2}]}, x^{[1]}-\delta x^{[4]}\rangle_{\FF_{q^6}}= \langle x^{[1]}, x^{[{2}]}, x^{[4]}\rangle_{\FF_{q^6}}.\]
The code $\overline{\cD}$ is not MRD, since it contains the polynomial $x^{[1]}-x^{[4]}$ which has kernel of dimension $3$, in particular it cannot be equivalent to $\cG_{3,1}$. It follows that $\ind(\cD_1)=2$ since $\langle  x^{[1]}, x^{[{2}]} \rangle_{\FF_{q^6}} \simeq \mathcal{G}_{2,1}$.
\smallskip

\begin{flushleft}
  \textbf{The code $\cD_2$}
\end{flushleft}

By Table~\ref{DkMRD} the code $\cD_2$ is
\[ \cD_2=\langle x^{[1]},x^{[3]},x-x^{[2]},x^{[4]}-\delta x \rangle_{\FF_{q^6}}, \]
with $q$ odd, $\delta^2+\delta=1$ and  $q \equiv 0,\pm 1 \pmod{5}$, hence $\delta \in \FF_q$.
Suppose $\ind(\cD_2)=3$, as before $\cD_2$ contains a code $\overline{\cD}$ equivalent to $\cG_{3,1}$. %As before, $\cD_2\cap \cD_2^{[1]} = \overline{\cD}\cap \overline{\cD}^{[1]}$; so
%\[ \overline{\cD}\cap \overline{\cD}^{[1]}=\langle -x^{[1]}+x^{[3]}, x^{[4]}-\delta x^{[2]} \rangle_{\FF_{q^6}}. \]
%Clearly, $\overline{\cD}^{[1]} \supseteq (\overline{\cD}\cap \overline{\cD}^{[1]})^{[1]}=\overline{\cD}^{[1]}\cap \overline{\cD}^{[2]}=\langle -x^{[2]}+x^{[4]}, x^{[5]}-\delta x^{[3]} \rangle_{\FF_{q^6}}$.
%So,
%\[ \overline{\cD}^{[1]}= \langle -x^{[1]}+x^{[3]}, x^{[4]}-\delta x^{[2]}, -x^{[2]}+x^{[4]} \rangle_{\FF_{q^6}}, \]
%and
Arguing as the previous case, we get
\[ \overline{\cD}= \langle -x+x^{[2]}, x^{[3]}-\delta x^{[1]}, -x^{[1]}+x^{[3]} \rangle_{\FF_{q^6}}= \langle x^{[1]}, x^{[3]}, -x+x^{[2]} \rangle_{\FF_{q^6}}. \]
To show that $\overline{\cD}$ is not equivalent to any $\cG_{3,s}$ we
compute its right idealiser $R(\overline{\cD})$.
Write $\displaystyle \varphi (x)=\sum_{i=0}^5 a_ix^{[i]} \in R(\overline{\cD})$;
then $x^{[1]}\circ \varphi(x), x^{[3]}\circ \varphi(x) \in \overline{\cD}$, so $\varphi(x)=\eta x$, for some $\eta \in \FF_{q^6}$.
Furthermore, $(x-x^{[2]}) \circ \varphi(x) \in \overline{\cD}$;
so $\eta=\eta^{[2]}$ and $\eta \in \FF_{q^2}$.
So, we get
% that $\varphi(x) \in R(\overline{\cD})$ if and only if $\varphi(x)=\eta x$ with $\eta \in \FF_{q^2}$,
$R(\overline{\cD}) \simeq \FF_{q^2}$.
If $\overline{\cD}$ were to be equivalent to $\cG_{3,1}$, by Proposition \ref{idealis} and by \cite[Corollary 5.2]{LTZ2}, it would follow that $R(\overline{\cD})$ is equivalent to $R(\cG_{3,1}) \simeq \FF_{q^6}$, which is not possible.
Suppose now
$\cD_2$ to contain a code $\overline{\cD}$ equivalent to $\cG_{2,1}$. Then by Theorem \ref{gabidulin-d} and by Lemma~\ref{Gablemma} we easily get $\overline{\cD}=\langle f(x), f(x)^{[1]} \rangle_{\FF_{q^6}}$ with $f(x)$ an invertible linearized polynomial.
% Since $\overline{\cD}\simeq \cG_{2,1}$, then
Also,
$\overline{\cD}\cap\overline{\cD}^{[1]}=\langle f(x)^{[1]} \rangle \subset \cD_2\cap\cD_2^{[1]}=\langle -x^{[1]}+x^{[3]}, x^{[4]}-\delta x^{[2]} \rangle_{\FF_{q^6}}$, so $f(x)^{[1]}=a(-x^{[1]}+x^{[3]})+b(x^{[4]}-\delta x^{[2]})$, since $f(x)$ is invertible we may assume $b=1$.
In particular, $\cD_2$ contains a code equivalent to $\cG_{2,1}$ if and only if there exists $a \in \FF_{q^6}$ such that $f(x)^{[1]}$ is invertible.
Let $D_{f^{[1]}}$ be the Dickson matrix associated to the polynomial $f(x)^{[1]}$
considered above. Then, for $a=1$
we have $\det D_{f^{[1]}} =16 (2-3\delta)\neq 0$.
%\footnote{\textcolor[rgb]{1.00,0.00,0.00}{This determinant has been evalueated with the help of \emph{Mathematica}}.}\neq 0$,
So,  $\cD_2$ contains $\langle -x-\delta x^{[1]} +x^{[2]}+x^{[3]} , -x^{[1]}+x^{[3]}+x^{[4]}-\delta x^{[2]}\rangle_{\FF_{q^6}}\simeq\cG_{2,1}$ and,
consequently, $\ind(\cD_2)=2$.
\smallskip

\begin{flushleft}
  \textbf{The code $\cD_4$}
\end{flushleft}

The code $\cD_4$ is
\[ \cD_4=\langle x^{[1]},x^{[2]},x^{[3]},x^{[5]},x^{[6]},x-\delta x^{[4]} \rangle_{\FF_{q^8}}, \]
with $q$ odd and $\delta^2=-1$.
Suppose that $\cD_4$ contains a code $\overline{\cD}$ equivalent to a generalized Gabidulin code of dimension $5$.
Since $\cG_{5,1} \simeq \cG_{5,7}$ and $\cG_{5,3} \simeq \cG_{5,5}$, we get that either $\overline{\cD}\simeq \cG_{5,1}$ or $\overline{\cD}\simeq \cG_{5,3}$.
By
Lemma~\ref{Gablemma},
%Result \ref{Gabidulincriteria},
$\dim (\overline{\cD}\cap \overline{\cD}^{[s]})=4$, with either $s=1$ or $s=3$, and, since $\cD_4$ is not equivalent to any generalized Gabidulin code, $\dim(\cD_4 \cap \cD_4^{[s]}) <5$, so $\cD_4\cap \cD_4^{[s]} = \overline{\cD}\cap \overline{\cD}^{[s]}$.
First assume that $\overline{\cD}\simeq \cG_{5,1}$.
It is easy to see that
\[\overline{\cD}\cap \overline{\cD}^{[1]}= \langle x^{[2]},x^{[3]},x^{[6]},x^{[1]}-\delta^{[1]} x^{[5]} \rangle_{\FF_{q^8}}.\]
Since the dimension of $\overline{\cD}$ is $5$ and $x^{[1]} \in \overline{\cD} \setminus (\overline{\cD}\cap \overline{\cD}^{[1]})$, it follows that
\[ \overline{\cD}=\langle x^{[1]}, x^{[2]},x^{[3]},x^{[6]},x^{[1]}-\delta^{[1]} x^{[5]} \rangle_{\FF_{q^8}}= \langle  x^{[1]}, x^{[2]}, x^{[3]}, x^{[5]}, x^{[6]} \rangle_{\FF_{q^8}}. \]
The Delsarte dual  of $\overline{\cD}$ is
\[ \overline{\cD}^\perp =\langle x,x^{[4]},x^{[7]} \rangle_{\FF_{q^8}}, \]
which is not MRD, since of $x-x^{[4]}$ has kernel of dimension $4$. By Lemma \ref{dualMRD}, neither $\overline{\cD}$ is an MRD-code, a contradiction.
Now, assume $\overline{\cD}\simeq \cG_{5,3}$. As before,
\[\overline{\cD}\cap \overline{\cD}^{[3]}= \langle x^{[1]}, x^{[5]}, x^{[6]}, x-\delta x^{[4]} \rangle_{\FF_{q^8}}\]
and
\[\overline{\cD}^{[5]}\cap \overline{\cD}= \langle x^{[6]}, x^{[2]}, x^{[3]}, x^{[5]}-\delta^{[5]} x^{[1]} \rangle_{\FF_{q^8}}.\]
So,
\[ \overline{\cD}=\langle x^{[1]}, x^{[6]}, x^{[2]}, x^{[3]}, x^{[5]}-\delta^{[5]} x^{[1]} \rangle_{\FF_{q^8}}= \langle x^{[1]}, x^{[2]}, x^{[3]}, x^{[5]}, x^{[6]}\rangle_{\FF_{q^8}}. \]

\noindent Again we get a contradiction since $\overline{\cD}$ is not an MRD-code.\\

Suppose now that $\cD_4$ contains a code $\overline{\cD}$ equivalent to $\cG_{4,1}$. By Theorem \ref{gabidulin-d} and by
Lemma~\ref{Gablemma},
 $\overline{\cD}=\langle p(x),p(x)^{[1]},p(x)^{[2]}, p(x)^{[3]} \rangle_{\FF_{q^8}}$ for some invertible linearized polynomial $p(x)\in \cD_4$. Clearly, $\langle p(x)^{[1]},p(x)^{[2]},p(x)^{[3]} \rangle_{\FF_{q^8}} \subset \langle x^{[2]},x^{[3]},x^{[6]},x^{[1]}-\delta^{[1]} x^{[5]} \rangle_{\FF_{q^8}}=\cD_4\cap\cD_4^{[1]}$ and so there exist $a,b,c,d \in \FF_{q^8}$ such that
\[p(x)^{[1]}= ax^{[2]}+bx^{[3]}+cx^{[6]}+d(x^{[1]}-\delta^{[1]} x^{[5]}),\]
\[p(x)^{[2]}=a^{[1]} x^{[3]}+b^{[1]}x^{[4]}+c^{[1]}x^{[7]}+d^{[1]}(x^{[2]}-\delta^{[2]} x^{[6]}),\]
\[p(x)^{[3]}=a^{[2]} x^{[4]}+b^{[2]}x^{[5]}+c^{[2]}x+d^{[2]}(x^{[3]}-\delta^{[3]} x^{[7]}).\]

Since these are all elements of $\cD_4$, we get $a=b=c=d=0$, i.e. $\cD_4$ cannot contain a code equivalent to $\cG_{4,1}$. Finally, suppose that $\overline{\cD}$ is equivalent to $\cG_{4,3}$. By Theorem \ref{gabidulin-d} and by Lemma~\ref{Gablemma},
$\overline{\cD}=\langle p(x),p(x)^{[3]},p(x)^{[6]}, p(x)^{[1]} \rangle_{\FF_{q^8}}$ for some invertible linearized polynomial $p(x)\in \cD_4$ and arguing as before we get a contradiction, i.e. $\cD_4$ cannot contain a code equivalent to $\cG_{4,3}$. So, $\cD_4$ cannot contain a code equivalent to a generalized Gabidulin code of dimension $4$ and so $\ind(\cD_4)<4$.
Since $\langle x^{[1]},x^{[2]},x^{[3]}\rangle_{\FF_{q^8}}\simeq \mathcal{G}_{3,1}$, it follows $\ind(\cD_4)=3$.
\end{proof}

%\textcolor[rgb]{1.00,0.00,0.00}{

Thus Theorem \ref{mth1} provides the following structure result on $k$-dimensional $\FF_{q^n}$-linear RM-codes with $h(\cC)=k-2$.
\begin{theorem}\label{mth2}
  Let $\cC$ be a $k$-dimensional $\FF_{q^n}$-linear RM-code of $\cL_{n,q}$ having $h(\cC)=k-2$, with $k>2$.
  Denote by $s$ an integer such that $\gcd(s,n)=1$ and $\dim(\cC\cap\cC^{[s]})=k-2$.
  Let $V:=\cC \cap \cC^{[s]}$ and suppose that $\cC \cap \mathcal{U}_1 =\{0\}$, then $\cC$ has one of the following forms
\begin{enumerate}
  \item if $\dim (V \cap V^{[s]})=k-3$, then there exist $p(x)$ and $q(x)$ in $\cC$ such that
        \[ \cC = \langle p(x), p(x)^{[s]}, \ldots, p(x)^{[{s(k-2)}]} \rangle_{\FF_{q^n}} \oplus \langle q(x) \rangle_{\FF_{q^n}}; \]
  \item if $\dim (V \cap V^{[s]})=k-4$, then there exist $p(x)$ and $q(x)$ in $\cC$ such that
        \[ \cC = \langle p(x), p(x)^{[s]}, \ldots, p(x)^{[{s(i-1)}]} \rangle_{\FF_{q^n}} \oplus \langle q(x), q(x)^{[s]}, \ldots, q(x)^{[{s(j-1)}]} \rangle_{\FF_{q^n}}, \]
        where $i+j=k$ and $i,j \geq 2$.
\end{enumerate}
  In particular, $\cC$ is equivalent to $\cH_{k,s}(\eta)$, for some $\eta \in \FF_{q^n}$, if and only if $\dim (V \cap V^{[s]})=k-3$, $p(x)$ is invertible and there exists $\eta \in \FF_{q^n}^*$ such that $p(x)+\eta p(x)^{[{sk}]} \in \cC$ and $\Nr_{q^n/q}(\eta)\neq (-1)^{kn}$.
\end{theorem}

\begin{remark}
Note that, in the hypothesis of Theorem \ref{mth2}, if the polynomials $p(x)$ and $q(x)$ are invertible,
then
either $\ind(\cC)=\dim\cC-1$ or $\displaystyle \ind(\cC)\geq\frac{\dim\cC}2$.
This holds for the known MRD-codes listed in the Tables \ref{kMRD} and \ref{DkMRD}; it is currently an open question whether an $\FF_{q^n}$-linear MRD-code $\cC$ having $h(\cC)=\dim\cC-2$ and $\ind(\cC)<\frac{\dim\cC}{2}$
might exist or not.
We also remark that  the known MRD-codes presented in the Tables \ref{kMRD} and \ref{DkMRD} which are not equivalent to a generalized Gabidulin code, have
$h(\cC)=\dim\cC-2$.

Suppose a code $\cC$ has generator matrix in standard form
$[I_k|X]$. Using the arguments of~\cite[Lemma 19]{H-TNRR} it can be
seen that $\dim(\cC\cap\cC^{[s]})\geq \dim\cC-i$ with $i>0$ if and only
if $rk(X-X^{[s]})\leq i$, and this condition
can be expressed by imposing that all minors of $X-X^{[s]}$ of rank $j>i$
have determinant $0$. In particular, the set of all codes
with $h(\cC)\geq\dim(\cC)-i$ is contained in the union of a finite number of
closed Zariski sets.
So, for a generic MRD-code we have $h(\cC)\in \max\{0,2k-n\}$.
We leave as an open problem to determine some
families of MRD-codes with $h(\cC)<\dim(\cC)-2$ and, more in detail,
to determine the possible spectrum of the values of $h(\cC)$ might
attain as $\cC$ varies among all MRD-codes over a given field.
\end{remark}

\section*{Acknowledgements}
  We thank an anonymous referee of the paper for having suggested
  a short and elegant proof for Lemma~\ref{fixedspace}.

\vskip.2cm
\noindent
\begin{minipage}[t]{\textwidth}
Authors' addresses:
\vskip.2cm\noindent\nobreak
\centerline{
\begin{minipage}[t]{7cm}
Luca Giuzzi\\
D.I.C.A.T.A.M. {\small (Section of Mathematics)} \\
University of Brescia\\
Via Branze 43, I-25123, Brescia, Italy \\
luca.giuzzi@unibs.it
\end{minipage}
\begin{minipage}[t]{7.5cm}
Ferdinando Zullo\\
Department of Mathematics and Physics \\
University of Campania ``\emph{Luigi Vanvitelli}'' \\
Viale Lincoln 5, I-81100, Caserta, Italy \\
ferdinando.zullo@unicampania.it
\end{minipage}
}

\end{minipage}

\end{document}